\documentclass[oneside,a4paper,reqno]{amsart}
\usepackage{pdfsync, enumitem}
\usepackage{amsmath,bm}
\usepackage{mathtools}
\usepackage{graphicx}
\usepackage[all,cmtip]{xy}
\usepackage{tikz-cd}
\usepackage{yfonts}

\newtheorem{introtheorem}{Theorem}

\usepackage[T1]{fontenc}
\usepackage[utf8]{inputenc}

\usepackage{etoolbox}
\patchcmd{\subsection}{\bfseries}{\itshape}{}{}

\usepackage{stmaryrd}
\usepackage{mathrsfs}
\usepackage{hyperref}
\usepackage{tikz-cd}
\usepackage{amssymb}
\usepackage{stackrel}
\usepackage{adjustbox}
\usepackage{multicol}
\usepackage{amsmath, amsthm, amscd, amssymb, latexsym, eucal}
\usepackage[all]{xy}
\def\serieslogo@{} \def\@setcopyright{} \makeatother

\usepackage[colorinlistoftodos]{todonotes}

\usepackage{hyperref}
\usepackage{color}
\usepackage{cite}
\usepackage{quiver}
\usepackage{marvosym}

\makeatletter
\renewcommand*\env@matrix[1][c]{\hskip -\arraycolsep
	\let\@ifnextchar\new@ifnextchar
	\array{*\c@MaxMatrixCols #1}}
\makeatother

\usepackage{color}

\numberwithin{equation}{section}
\newtheorem{thm}{Theorem}[section]

\newtheorem*{Auslander-thm}{Auslander's Theorem}

\newtheorem{cor}[thm]{Corollary}
\newtheorem{lem}[thm]{Lemma}
\newtheorem{prop}[thm]{Proposition}

\newtheorem*{thmC}{Theorem~C}
\newtheorem*{corB}{Corollary B}

\theoremstyle{definition}

\newtheorem{rem}[thm]{Remark}

\newtheorem{reminder}[thm]{Reminder}
\newtheorem{exmp}[thm]{Example}

\newcommand\padova{%
\mathrel{\ooalign{\hss{\scalebox{0.5}{$\longleftrightarrow$}}\hss\cr%
\kern0.9ex\raise0.55ex\hbox{\scalebox{0.5}{$\boldsymbol{\bm{\vert}}$}}}}}
\newcommand\hfpadova{%
\mathrel{\ooalign{\hss{\scalebox{0.5}{$\longleftrightarrow$}}\hss\cr%
\kern0.9ex\raise0.55ex\hbox{\scalebox{0.5}{$\boldsymbol{\mathsf{h}\mathsf{f}\ \ }$}}}}}
\newcommand\rpadova{%
\mathrel{\ooalign{\hss{\scalebox{0.5}{$\longrightarrow$}}\hss\cr%
\kern0.9ex\raise0.55ex\hbox{\scalebox{0.5}{$\boldsymbol{\bm{\vert}}$}}}}}
\newcommand\lpadova{%
\mathrel{\ooalign{\hss{\scalebox{0.5}{$\longleftarrow$}}\hss\cr%
\kern0.9ex\raise0.55ex\hbox{\scalebox{0.5}{$\boldsymbol{\bm{\vert}}$}}}}}

\newcommand{\X}{\mathcal X}

\newtheorem{theoremprime}{Theorem}

\newcommand{\heart}{\ensuremath\heartsuit}

\DeclareMathOperator*{\opp}{\mathsf{op}}

\DeclareMathOperator{\pd}{\mathsf{pdim}}

\DeclareMathOperator*{\Findim}{\mathsf{Fin.\!dim}}

\DeclareMathOperator*{\Mod}{\mathsf{Mod}\!}

\DeclareMathOperator*{\smod}{\mathsf{mod}\!}

\DeclareMathOperator*{\umod}{\underline{\mathsf{mod}}\!}

\DeclareMathOperator*{\proj}{\mathsf{proj}\!}

\DeclareMathOperator*{\Inj}{\mathsf{Inj}\!}

\DeclareMathOperator*{\Proj}{\mathsf{Proj}\!}

\usepackage{stackengine}

\newsavebox{\proofbox}
\savebox{\proofbox}{\begin{picture}(7,7)%
	\put(0,0){\framebox(7,7){}}\end{picture}}

\usepackage{latexsym}
\usepackage{pstricks}
\usepackage{comment}

\usepackage[greek,english]{babel}

\begin{document}

\thanks{\textbf{Acknowledgements:} I thank Jan Šťovíček for telling me about the different approach in \cite{dey stovicek}. I am grateful to Chrysostomos Psaroudakis for comments and for many discussions on the finitistic dimension over the years.}

\title{Finitistic dimension via modules over the singularity category}

\author[Kostas]{Panagiotis Kostas}
\address{Dipartimento di Matematica ``F.~Enriques'', Universita degli Studi di Milano, Via Cesare
Saldini 50, 20133 Milano, Italy}
\email{panagiotis.kostas@unimi.it}

\begin{abstract}
    We combine results of Rickard and Shaul with methods of intrinsic homological algebra and the theory of purity to prove that the finiteness of the big finitistic dimension of an Artin algebra is an intrinsic property of the singularity category of its opposite algebra, via its category of modules. This `object-free' approach complements work of Dey--Šťovíček and arrives at the same conclusion: the finiteness of the finitistic dimension of Artin algebras is a singular invariant (of the opposite algebras). In fact, our characterisation can be used to prove that finite finitistic dimension descends along certain fully faithful functors between singularity categories. We include an appendix where we explain how our methods can be used to prove non-existence of bounded t-structures on singularity categories. 
\end{abstract}
\maketitle
\setcounter{tocdepth}{1}

\section{Introduction}
The finitistic dimension conjecture (FDC) is a longstanding open problem in the representation
theory of finite dimensional algebras, see the survey \cite{survey} for details. It was not until very recently that the finiteness of the big finitistic dimension was proven to be a property of the singularity category, a result due to Dey and Šťovíček announced in \cite{stovicek} and proved in work in preparation \cite{dey stovicek}. In this paper we establish a new criterion to detect the finiteness of the finitistic dimension via the singularity category. Specifically, we realise the finiteness of the finitistic dimension of an Artin algebra $\Lambda$ as an intrinsic property of $\mathsf{D}_{\mathsf{sg}}(\Lambda^{\opp})$, by looking at the category of additive functors $\mathsf{D}_{\mathsf{sg}}(\Lambda^{\opp})^{\opp}\rightarrow\mathsf{Ab}$, i.e.\ the category of \emph{modules} over the singularity category. Our approach differs substantially from that of \cite{dey stovicek}, where the finiteness of $\Findim\Lambda$ is characterised by the existence of a `generator' in $\mathsf{D}_{\mathsf{sg}}(\Lambda^{\opp})$ satisfying certain properties. Our main result, presented below, does not rely on the existence of specific objects.

\begin{introtheorem} \label{thmA}
     Let $\Lambda$ be an Artin algebra and consider $\mathcal{D}=\mathsf{D}_{\mathsf{sg}}(\Lambda^{\opp})$. Then it holds that $\Findim\Lambda<\infty$ if and only if $Q^{\perp_{\gg}}\cap \mathsf{Inj}\mathcal{D}=0$ in $\mathsf{Mod}\mathcal{D}$, where $Q$ denotes the image of $\mathcal{D}$ along Yoneda. 
\end{introtheorem}
In the above theorem, $Q^{\perp_{\gg}}$ denotes a certain subcategory of $\mathsf{Mod}\mathcal{D}$ (a `far-away orthogonal' of $Q$). This is defined by using the autoequivalence on $\mathsf{Mod}\mathcal{D}$ induced by the suspension of $\mathcal{D}$, which we always consider as part of its structure. In simpler terms, the theorem above asks that there is no nonzero additive functor $F\colon \mathsf{D}_{\mathsf{sg}}(\Lambda^{\opp})^{\opp}\rightarrow \mathsf{Ab}$ that is injective in $\Mod \mathcal{D}$ and moreover, for every finitely generated $\Lambda^{\opp}$-module $x$, it holds that $F(x[-n])=0$ for $n\gg 0$ (depending on $x$). In fact, it is enough to test the last condition only on the simple $\Lambda^{\opp}$-modules and specifically for $x=\Lambda^{\opp}/\mathsf{rad}\Lambda^{\opp}$, see Remark~\ref{remark} for details.  

The proof of Theorem \ref{thmA} combines arguments and results from \cite{krause, krause fin dim, rickard,shaul} with a crucial computation from \cite{KPV}. Since the finiteness of the finitistic dimension is an intrinsic property of the singularity category of the opposite algebra, it is also a singular invariant (see Corollary~\ref{singular invariant}). As already mentioned, we prove a stronger statement about fully faithful functors (see Theorem \ref{ff}).

In the statement of Theorem \ref{thmA}, one may replace injective modules by cohomological functors. We include this reformulation below, for which we use the terminology of Rouquier \cite{rouquier} on bounded below cohomological functors.

\renewcommand{\thetheoremprime}{A'}
\begin{theoremprime}
\label{thmA'}
     Let $\Lambda$ be an Artin algebra and consider $\mathcal{D}=\mathsf{D}_{\mathsf{sg}}(\Lambda^{\opp})$. Then it holds that $\Findim\Lambda<\infty$ if and only if there is no nonzero cohomological bounded below functor $\mathcal{D}^{\opp}\rightarrow \mathsf{Ab}$ (in other words $Q^{\perp_{\gg}}\cap \mathsf{Coh}\mathcal{D}=0$).  
\end{theoremprime}

The proofs identify $\mathsf{D}_{\mathsf{sg}}(\Lambda^{\opp})$ as the compact objects of $\mathcal{T}=\mathsf{K}_{\mathsf{ac}}(\Inj\Lambda^{\opp})$ and it is useful to summarise the main implications below. 
\[\begin{tikzcd}
	{\Findim\Lambda<\infty} &&&& \\
	\\
	{\mathcal{T}^-\cap \mathsf{PInj}\mathcal{T}=0} && {Q^{\perp_{\gg}}\cap \mathsf{Inj}\mathcal{T}^{\mathsf{c}}=0} && {Q^{\perp_{\gg}}\cap \mathsf{Coh}\mathcal{T}^{\mathsf{c}}=0}
	\arrow["{\textnormal{Prop. }\ref{prop: findim via pure injectives}}"{description}, Rightarrow, 2tail reversed, from=1-1, to=3-1]
	\arrow["{\textnormal{Thm } \ref{thmA}}"{description}, curve={height=-12pt}, tail reversed, from=1-1, to=3-3]
	\arrow["{\textnormal{Thm }\ref{thmA'}}"{description}, curve={height=-18pt}, tail reversed, from=1-1, to=3-5]
	\arrow["{\textnormal{Lem. }\ref{lem: bounded above pure}}", Rightarrow, 2tail reversed, from=3-1, to=3-3]
	\arrow["{\textnormal{Lem. }\ref{lem: coh vs inj}}", Rightarrow, 2tail reversed, from=3-3, to=3-5]
\end{tikzcd}\]

By Theorem \ref{thmA'} (applied to $\Lambda^{\opp}$) combined with work of Cummings \cite{cummings}, it follows that the validity of the FDC depends only on the singularity categories of a very restricted class of algebras (see Remark \ref{cummings} for more details).

\begin{corB} 
    The FDC holds if and only if for every finite dimensional algebra $\Lambda$ with $\Findim\Lambda=0$, there is no nonzero cohomological functor $F\colon \mathsf{D}_{\mathsf{sg}}(\Lambda)^{\opp}\rightarrow \mathsf{Ab}$ satisfying $F(\Omega^n(\Lambda/\mathsf{rad}\Lambda))=0$ for $n\gg0$.
\end{corB}

The non-existence of bounded t-structures is a topic that has seen tremendous progress in recent years \cite{K theoretic,bounded t-structures, neeman}. Theorem \ref{thmA'} is particularly interesting, since it can be used to prove the non-existence of certain t-structures on the singularity category. We explain this idea further in an appendix and in particular prove the following. 

\begin{thmC}[Corollary \ref{cor1} \& Corollary \ref{cor2}]
    Let $R$ be a noetherian ring satisfying the equality 
    \[
    \mathsf{D}^{\mathsf{b}}(\Mod R)=\mathsf{broad}_{\Sigma}(\mathsf{D}^{\mathsf{b}}(\smod R)).
    \]
    Then the singularity category $\mathsf{D}_{\mathsf{sg}}(R)$, assuming that it is nonzero, does not admit a bounded t-structure. This applies in particular to commutative noetherian quasiexcellent rings of finite Krull dimension.  
\end{thmC}

We refer to the Appendix for the explanation of the `broad' symbol appearing in the statement. Although the above is already known for Artin algebras (see Remark \ref{rem: singularity of Artin}), the conclusion about quasiexcellent rings seems to be new.

\section{The proofs of the theorems}
In this section we provide the proofs of the theorems stated in the Introduction, for which we collect several preliminaries. Throughout we fix an Artin algebra $\Lambda$ over a commutative artinian ring $k$. By module we always mean a right module. We agree on the following notation 

\begin{itemize}[labelindent=\parindent,leftmargin=*,label=(MPT\arabic*)]
    \item[$\Mod\Lambda$] \hspace{0.4cm} the category of $\Lambda$-modules;
    \item[$\smod \Lambda$] \hspace{0.4cm} the category of finitely generated $\Lambda$-modules;
    \item[$\Proj\Lambda$] \hspace{0.4cm} the category of projective $\Lambda$-modules;
    \item[$\Inj\Lambda$] \hspace{0.4cm} the category of injective $\Lambda$-modules; 
    \item[$\proj\Lambda$] \hspace{0.4cm} the category of finitely generated projective $\Lambda$-modules.
\end{itemize} 
For an additive category $\mathcal{X}$, we further consider 
\begin{itemize}[labelindent=\parindent,leftmargin=*,label=(MPT\arabic*)]
    \item[$\mathsf{C}(\mathcal{X})$] \hspace{0.4cm} the category of (cochain) complexes of $\mathcal{X}$;
    \item[$\mathsf{K}(\mathcal{X})$] \hspace{0.4cm} the homotopy category of complexes of $\mathcal{X}$;
    \item[$\mathsf{K}^{\mathsf{b}}(\mathcal{X})$] \hspace{0.4cm} the subcategory of $\mathsf{K}(\mathcal{X})$ that consists of bounded complexes. 
\end{itemize}
We set $\mathsf{C}(\Lambda)=\mathsf{C}(\Mod\Lambda)$. Moreover, we will use the following derived categories
\begin{itemize}[labelindent=\parindent,leftmargin=*,label=(MPT\arabic*)]
    \item[$\mathsf{D}(\Lambda)$] \hspace{0.4cm} the derived category of $\Mod \Lambda$; 
    \item[$\mathsf{D}^{\mathsf{b}}(\smod \Lambda)$] \hspace{0.4cm} the bounded derived category of $\smod\Lambda$. 
\end{itemize}
Lastly, we denote by $D$ the Matlis duality $\mathsf{Hom}_{k}(-,E)\colon\Mod \Lambda\rightarrow\Mod\Lambda^{\opp}$, where $E$ is a minimal injective cogenerator over $k$. It satisfies that $D(\Proj\Lambda)\subseteq \Inj\Lambda^{\opp}$. In the same spirit, for any category $\mathcal{C}$ and any functor $F\colon \mathcal{C}\rightarrow \mathsf{Ab}$, we denote by $F^+$ its character dual, that is  $\mathsf{Hom}_{\mathbb{Z}}(F(-),\mathbb{Q}/\mathbb{Z})\colon \mathcal{C}^{\opp}\rightarrow \mathsf{Ab}$.

\subsection{Compact generation} Let $\mathcal{T}$ be a triangulated category that has set-indexed coproducts. An object $c$ of $\mathcal{T}$ is \emph{compact} if $\mathsf{Hom}_{\mathcal{T}}(c,-)$ commutes with coproducts. We denote by $\mathcal{T}^{\mathsf{c}}$ the subcategory of compact objects and recall that $\mathcal{T}$ is \emph{compactly generated} if $\mathcal{T}^{\mathsf{c}}$ is skeletally small and $(\mathcal{T}^{\mathsf{c}})^{\perp}=0$. 

\subsection{Singularity categories} We denote by $\mathsf{K}_{\mathsf{ac}}(\Inj\Lambda)$ the subcategory of $\mathsf{K}(\Inj\Lambda)$ that consists of acyclic complexes. Moreover, we use $\mathsf{D}_{\mathsf{sg}}(\Lambda)$ to denote the \emph{singularity category} of $\Lambda$, i.e.\ the Verdier quotient $\mathsf{D}^{\mathsf{b}}(\smod \Lambda)/\mathsf{K}^{\mathsf{b}}(\proj\Lambda)$.

\begin{lem} \label{lem: big sing is cg}
    The triangulated categories $\mathsf{K}(\Inj\Lambda)$ and  $\mathsf{K}_{\mathsf{ac}}(\Inj\Lambda)$ are compactly generated. Moreover, it holds that $\mathsf{K}_{\mathsf{ac}}(\Inj\Lambda)^{\mathsf{c}}\simeq \mathsf{D}_{\mathsf{sg}}(\Lambda)$. 
\end{lem}
\begin{proof}
    By \cite[Proposition 2.3]{krause stable} and \cite[Corollary 5.4]{krause stable} we know that $\mathsf{K}(\Inj\Lambda)$ and $\mathsf{K}_{\mathsf{ac}}(\Inj\Lambda)$ are compactly generated. Moreover, the subcategory of compact objects of the latter is the idempotent completion of $\mathsf{D}_{\mathsf{sg}}(\Lambda)$. Since $\Lambda$ is an Artin algebra, we know by \cite[Corollary 2.4]{chen} that its singularity category is idempotent complete, which yields the second claim. 
\end{proof}

\subsection{Far-away orthogonals} \label{far away} We follow \cite{KPV} and \cite[Appendix B]{krause fin dim} and for any additive category $\mathcal{C}$ equipped with an equivalence $\Sigma\colon \mathcal{C}\to \mathcal{C}$ and, for any collection $\X$ of objects in $\mathcal{C}$, we consider the full subcategory
\[
{\X}^{\perp_{\gg}}=\{y\in \mathcal{C}: \forall x\in\X,\mathsf{Hom}_{\mathcal{C}}(x,\Sigma^ny)= 0 \text{ for }n\gg 0\}.
\]
In particular, when $\mathcal{T}$ is a triangulated category and $\Sigma=[1]$, then $(\mathcal{T}^{\mathsf{c}})^{\perp_{\gg}}\eqqcolon \mathcal{T}^{-}$ is a thick subcategory of $\mathcal{T}$, which we call the subcategory of \emph{bounded above} objects of $\mathcal{T}$, see \cite{KPV}. The following result was shown in \cite[Proposition 3.21]{KPV} (see also the proof of Proposition \ref{no bounded objects} where the main ideas are recalled).

\begin{lem} \label{lem: computation from KPV}
    Consider $\mathcal{T}=\mathsf{K}_{\mathsf{ac}}(\Inj\Lambda)$. Then $\mathcal{T}^-=\mathsf{K}^-_{\mathsf{ac}}(\Inj\Lambda)$, where $\mathsf{K}^-_{\mathsf{ac}}(\Inj\Lambda)$ is the subcategory of $\mathsf{K}_{\mathsf{ac}}(\Inj\Lambda)$ that consists of the bounded above complexes.
\end{lem}

Recall that the \emph{big finitistic dimension} of $\Lambda$ is defined as 
\[
\Findim\Lambda\coloneqq \mathsf{sup}\{\pd x:x\in\Mod\Lambda, \pd x<\infty\}.
\]
The \emph{`big' finitistic dimension conjecture} (which we simply refer to as FDC) asserts that $\Findim\Lambda<\infty$ whenever $\Lambda$ is a finite dimensional algebra over a field. As proved in \cite[Theorem 5.1]{shaul}, the bounded above acyclic complexes of injectives control the finiteness of the finitistic dimension, in the following sense. 

\begin{prop} \label{prop: theorem of shaul}
    It holds that $\Findim\Lambda<\infty$ if and only if $\mathsf{K}^-_{\mathsf{ac}}(\Inj\Lambda^{\opp})=0$. 
\end{prop}
\begin{proof}[A comment] \renewcommand{\qedsymbol}{}
     Shaul shows in \cite[Theorem 5.1 c)]{shaul} that the above is true for any two-sided noetherian ring with a dualizing complex. An Artin algebra $\Lambda$ satisfies these assumptions (with $D(\Lambda)$ being the dualizing complex in this case). In a way we reprove the `hard' implication of Proposition \ref{prop: theorem of shaul} in Proposition \ref{prop: findim via pure injectives}, since we need a slightly stronger form.
\end{proof}

\begin{rem}
    By combining Lemma \ref{lem: computation from KPV} with Proposition \ref{prop: theorem of shaul}, it follows that the finiteness of $\Findim\Lambda$ is an intrinsic property of $\mathcal{T}=\mathsf{K}_{\mathsf{ac}}(\Inj\Lambda^{\opp})$ (namely measured by $\mathcal{T}^-=0)$; this is already new information about the finitistic dimension, but the point of this paper is to understand its `small' variant (i.e.\ via the ordinary singularity category), which requires some extra work.
\end{rem}

\subsection{Module categories} For any small triangulated category $\mathcal{D}$, \emph{the category of modules} over $\mathcal{D}$, denoted by $\mathsf{Mod}\mathcal{D}$, is the category of additive functors $\mathcal{D}^{\mathsf{op}}\rightarrow \mathsf{Ab}$. This is a Grothendieck category (see \cite[Chapter 5]{freyd}) and we use $\mathsf{Inj}\mathcal{D}$ to denote its injective objects. The suspension functor $[1]\colon\mathcal{D}\rightarrow\mathcal{D}$ extends (via Yoneda) to an equivalence 
\begin{equation} \label{shift of modules}
    \Sigma \colon \mathsf{Mod}\mathcal{D}\xrightarrow{\sim} \mathsf{Mod}\mathcal{D}
\end{equation}
given by $F\mapsto F\circ [-1]$.

\subsection{Purity} 

Let $\mathcal{A}$ be a locally finitely presented additive category \cite{Crawley-Boevey}, typically assumed to have products (for instance the category of (co)chain complexes over a ring). We denote by $\mathsf{fp}(\mathcal{A})$ its subcategory of finitely presented objects. A sequence of morphisms $0\to x\to y\to z\to 0$ in $\mathcal{A}$ is \emph{pure exact} if 
\[
0\rightarrow \mathsf{Hom}_{\mathcal{A}}(c,x)\rightarrow \mathsf{Hom}_{\mathcal{A}}(c,y)\rightarrow\mathsf{Hom}_{\mathcal{A}}(c,z)\rightarrow0
\]
is an exact sequence of abelian groups for any $c\in\mathsf{fp}(\mathcal{A})$. An object $q\in\mathcal{A}$ is \emph{pure-injective} if for any pure exact sequence as above, the induced homomorphism $\mathsf{Hom}_{\mathcal{A}}(y,q)\rightarrow \mathsf{Hom}_{\mathcal{A}}(x,q)$ is surjective. We will need the following characterisation of pure-injective objects, see \cite[Theorem 1]{Crawley-Boevey}. 

\begin{prop} \label{prop: pure injective in fp}
    An object $q$ in $\mathcal{A}$ is pure-injective if and only if for any set $I$, the summation map $\oplus_{I}q\to q$ factors through the canonical map $\oplus_{I}q\to \Pi_{I}q$.
\end{prop}

We now recall the analogous notions for a compactly generated triangulated category $\mathcal{T}$. A morphism $x\to y$ in $\mathcal{T}$ is a \emph{pure monomorphism} if for every compact object $c\in\mathcal{T}$, the induced homomorphism $\mathsf{Hom}_{\mathcal{T}}(c,x)\rightarrow\mathsf{Hom}_{\mathcal{T}}(c,y)$ is a monomorphism. Then, an object $x$ of $\mathcal{T}$ is called \emph{pure-injective} if every pure monomorphism $x\to y$ splits. The following is proved in \cite[Theorem 1.8]{krause}. 
\begin{prop} \label{prop: pure injective in cg}
    An object $x$ in $\mathcal{T}$ is pure-injective if and only if for any set $I$, the summation map $\oplus_{I}x\to x$ factors through the canonical map $\oplus_{I}x\to \Pi_{I}x$. 
\end{prop}

We now consider the \emph{restricted Yoneda functor} $\mathcal{T}\rightarrow \mathsf{Mod}\mathcal{T}^{\mathsf{c}}$, which is given by $t\mapsto H_t=\mathsf{Hom}_{\mathcal{T}}(-,t)|_{\mathcal{T}^{\mathsf{c}}}$. We know from \cite[Corollary 1.9]{krause} that the latter yields an equivalence
\begin{equation} \label{equivalence}
    \mathsf{PInj}\mathcal{T}\simeq \mathsf{Inj}\mathcal{T}^{\mathsf{c}}
\end{equation}
between the pure-injective objects of $\mathcal{T}$ and the injective objects of $\mathsf{Mod}\mathcal{T}^{\mathsf{c}}$. We will identify $\mathcal{T}^{\mathsf{c}}$ with a full subcategory $Q$ of $\mathsf{Mod}\mathcal{T}^{\mathsf{c}}$. The following idea is found in \cite[Appendix B]{krause fin dim} and is essential to this paper.

\begin{lem} \label{lem: bounded above pure}
    The restricted Yoneda functor $\mathcal{T}\rightarrow \mathsf{Mod}\mathcal{T}^{\mathsf{c}}$ identifies the pure-injective objects of $\mathcal{T}$ that are contained in $\mathcal{T}^-$ with $Q^{\perp_{\gg}}\cap \mathsf{Inj}\mathcal{T}^{\mathsf{c}}\subseteq\mathsf{Mod}\mathcal{T}^{\mathsf{c}}$.
\end{lem}
\begin{proof}
    For any object $t\in\mathcal{T}$ and any integer $n$, we have  $H_{t[n]}\cong H_t\circ[-n]\cong \Sigma^nH_t$ in $\Mod\mathcal{T}^{\mathsf{c}}$ and therefore for any compact object $c\in\mathcal{T}^{\mathsf{c}}$, it holds that
    \[
    \mathsf{Hom}_{\mathcal{T}}(c,t[n]) \cong  \mathsf{Hom}_{\Mod \mathcal{T}^{\mathsf{c}}}(H_c,H_{t[n]}) \cong \mathsf{Hom}_{\Mod\mathcal{T}^{\mathsf{c}}}(H_c,\Sigma^n H_t).
    \]
    Hence we see that $t$ belongs to $\mathcal{T}^-$ if and only if $H_t$ belongs to $Q^{\perp_{\gg}}\subseteq \Mod\mathcal{T}^{\mathsf{c}}$. It follows that the equivalence  $\mathsf{PInj}\mathcal{T}\simeq \Inj\mathcal{T}^{\mathsf{c}}$, induced by the restricted Yoneda functor (\ref{equivalence}), restricts to an equivalence $\mathcal{T}^-\cap \mathsf{PInj}\mathcal{T}\simeq Q^{\perp_{\gg}}\cap \Inj\mathcal{T}^{\mathsf{c}}$, as claimed. 
\end{proof}

\begin{lem} \label{lem: dual are pure-injective}
    For any complex $x$ of $\Lambda$-modules, the complex $D(x)$ is a pure-injective object of $\mathsf{C}(\Lambda^{\opp})$. In particular when $x$ is a complex of projective $\Lambda$-modules, then $D(x)$ is a pure-injective object of $\mathsf{K}(\Inj\Lambda^{\opp})$.
\end{lem}
\begin{proof}
    That  $D(x)$ is a pure-injective object in $\mathsf{C}(\Lambda^{\opp})$ is a well-known statement; see for instance the proof of \cite[Lemma B.3]{krause fin dim}. It follows by Proposition \ref{prop: pure injective in fp} that for any set $I$ the summation map $\oplus_ID(x)\to D(x)$ factors through $\oplus_ID(x)\to \Pi_{I}D(x)$. When $x$ is a complex of projective $\Lambda$-modules, then $D(x)$ is an object of $\mathsf{K}(\Inj\Lambda^{\opp})$ and it is pure-injective by the above together with Proposition \ref{prop: pure injective in cg} and the fact that $\mathsf{K}(\Inj\Lambda^{\opp})$ is compactly generated. 
\end{proof}
\vspace{0.1cm}
\begin{center}
    * \ \ \ \ \  * \ \ \ \ \  * 
\end{center}
The following is the key tool for the proof of Theorem \ref{thmA}. 

\begin{prop} \label{prop: findim via pure injectives}
    It holds that $\Findim\Lambda<\infty$ if and only if 
    \[
    \mathcal{T}^{-}\cap \mathsf{PInj}\mathcal{T}=0
    \]
    where $\mathcal{T}=\mathsf{K}_{\mathsf{ac}}(\Inj\Lambda^{\opp})$.
\end{prop}
\begin{proof}
    We know from Lemma \ref{lem: computation from KPV} that $\mathcal{T}^-=\mathsf{K}^-_{\mathsf{ac}}(\Inj \Lambda^{\opp})$.
    Therefore, if we assume that $\Findim\Lambda<\infty$, then the equality $\mathcal{T}^{-}=0$ is a consequence of Proposition~\ref{prop: theorem of shaul}. Assume now that $\mathcal{T}^-\cap \mathsf{PInj\mathcal{T}}=0$ and for the sake of contradiction that $\Lambda$ has infinite finitistic dimension. We follow a construction from \cite[Theorem 4.3]{rickard} (where the arguments are given for finite dimensional algebras over a field, but hold verbatim in our setup too); since it holds that $\Findim\Lambda=\infty$, we may find a sequence of modules $M_i$ with finite and strictly increasing projective dimension. We consider $P_i$ to be a minimal projective resolution of $M_i$ and we let $C$ be the cone of the canonical inclusion 
    \[
    \oplus P_i[-d_i]\to \Pi P_i[-d_i],
    \]
    where $d_i$ is the projective dimension of $M_i$. By \emph{loc.\ cit.}, the complex $C$ is a bounded below acyclic complex of projective $\Lambda$-modules that is not contractible. Therefore, the complex $D(C)$ is a bounded above complex of injective $\Lambda^{\opp}$-modules that is acyclic. We claim that it is not contractible; indeed, if it were, then $C\otimes_{\Lambda}x$ would be acyclic for any  complex $x$ of $\Lambda^{\opp}$-modules, since it holds that 
    \[
    D(C\otimes_{\Lambda}x)\cong \mathcal{H}om_{\Lambda^{\opp}}(x,D(C)) 
    \]
    and 
    \[
    H^n(\mathcal{H}om_{\Lambda^{\opp}}(x,D(C)))\cong \mathsf{Hom}_{\mathsf{K}(\Lambda^{\opp})}(x,D(C)[n]).
    \]
    However, the complex $C\otimes_{\Lambda}D(\Lambda)$ is not contractible (since $C$ itself is not contractible and the functor $-\otimes_{\Lambda}D(\Lambda)\colon \mathsf{K}(\Proj\Lambda)\rightarrow \mathsf{K}(\Inj\Lambda)$ is an equivalence \cite{iyengar krause}) and therefore not acyclic, since it is a bounded below complex of injectives (and such a complex is acyclic if and only if it is contractible). It remains to observe from Lemma~\ref{lem: dual are pure-injective} that $D(C)$ is pure-injective in $\mathsf{K}(\Inj\Lambda^{\opp})$ and thus also in $\mathsf{K}_{\mathsf{ac}}(\Inj\Lambda^{\opp})$. All in all, $D(C)$ is a nonzero object in $\mathcal{T}^-\cap \mathsf{PInj}\mathcal{T}$, contradicting our assumption.
\end{proof}

\begin{proof}[Proof of Theorem \ref{thmA}]
    We know from Lemma \ref{lem: big sing is cg} that $\mathcal{D}\simeq \mathcal{T}^{\mathsf{c}}$, where $\mathcal{T}$ denotes $\mathsf{K}_{\mathsf{ac}}(\Inj\Lambda^{\opp})$. Therefore, it follows from Lemma \ref{lem: bounded above pure} that $Q^{\perp_{\gg}}\cap \mathsf{Inj}\mathcal{D}\subseteq \mathsf{Mod}\mathcal{D}$ is identified with $\mathcal{T}^{-}\cap \mathsf{PInj}\mathcal{T}\subseteq \mathcal{T}$ and the claim follows from Proposition \ref{prop: findim via pure injectives}.
\end{proof}

\begin{cor} \label{singular invariant} (Dey--Šťovíček)
    Let $\Gamma,\Lambda$ be Artin algebras for which there is a triangle equivalence $\mathsf{D}_{\mathsf{sg}}(\Gamma^{\opp})\simeq \mathsf{D}_{\mathsf{sg}}(\Lambda^{\opp})$. Then it holds that $\Findim\Gamma<\infty$ if and only if $\Findim\Lambda<\infty$. 
\end{cor}
\begin{proof}
    This is a direct consequence of intrinsicness, but let us nonetheless spell out the argument. We write $\mathcal{D}_{\Lambda}$ for $\mathsf{D}_{\mathsf{sg}}(\Lambda^{\opp})$, $\mathcal{D}_{\Gamma}$ for $\mathsf{D}_{\mathsf{sg}}(\Gamma^{\opp})$ and $\mathcal{Q}_{\Lambda}$ and $\mathcal{Q}_{\Gamma}$ for their images along Yoneda, respectively. A triangle equivalence of singularity categories $\mathsf{D}_{\mathsf{sg}}(\Gamma^{\opp})\simeq \mathsf{D}_{\mathsf{sg}}(\Lambda^{\opp})$ extends to an equivalence $\mathsf{Mod}\mathcal{D}_{\Gamma}\simeq \mathsf{Mod}\mathcal{D}_{\Lambda}$ which commutes with the autoequivalences of module categories defined in (\ref{shift of modules}). The above therefore restricts to an equivalence $Q_{\Gamma}^{\perp_{\gg}}\cap \mathsf{Inj}\mathcal{D}_{\Gamma}\simeq Q_{\Lambda}^{\perp_{\gg}}\cap \mathsf{Inj}\mathcal{D}_{\Lambda}$, which completes the proof. 
\end{proof}

\begin{center}
    * \ \ \ \ \  * \ \ \ \ \ *
\end{center}

\begin{reminder}
    For any $\Lambda$-module $x$ we denote by $\Omega(x)$ the kernel of a projective cover $P\rightarrow x$. This gives rise to a functor $\Omega\colon \umod\Lambda\rightarrow \umod\Lambda$, where $\umod\Lambda$ denotes the stable module category of $\Lambda$, and we set inductively $\Omega^n\coloneqq \Omega^{n-1}\circ \Omega$ for $n\geq 2$. The functor  $\smod\Lambda\rightarrow \mathsf{D}_{\mathsf{sg}}(\Lambda)$ induces a functor $\umod\Lambda\rightarrow \mathsf{D}_{\mathsf{sg}}(\Lambda)$ and the following diagram commutes 
    \[\begin{tikzcd}
	{\umod\Lambda} && {\umod\Lambda} \\
	{\mathsf{D}_{\mathsf{sg}}(\Lambda)} && {\mathsf{D}_{\mathsf{sg}}(\Lambda)}
	\arrow["{\Omega^n}", from=1-1, to=1-3]
	\arrow[from=1-1, to=2-1]
	\arrow[from=1-3, to=2-3]
	\arrow["{[-n]}"', from=2-1, to=2-3]
\end{tikzcd}\]
    see \cite[Lemma 2.2]{chen}.
\end{reminder}

\begin{rem} \label{remark}
    For every object $y\in\mathsf{D}_{\mathsf{sg}}(\Lambda^{\opp})$ and any $F\in\mathsf{Mod}\mathcal{D}$, it holds that 
    \begin{align*}
        \mathsf{Hom}_{\mathsf{Mod}\mathcal{D}}(\mathsf{Hom}_{\mathcal{D}}(-,y),\Sigma^nF)&\cong \mathsf{Hom}_{\mathsf{Mod}\mathcal{D}}(\mathsf{Hom}_{\mathcal{D}}((-) [n],y),F) \\ 
        & \cong \mathsf{Hom}_{\mathsf{Mod}\mathcal{D}}(\mathsf{Hom}_{\mathcal{D}}(-,y[-n]),F) \\ 
        & \cong F(y[-n]).
    \end{align*}
    Therefore, $F$ belongs to $Q^{\perp_{\gg}}$ if and only if for every object $y\in\mathsf{D}_{\mathsf{sg}}(\Lambda^{\opp})$ it holds that $F(y[-n])= 0$ for $n\gg 0$ (depending on $y$). Now, since every such $y$ is isomorphic to $x[k]$ for some finitely generated $\Lambda^{\opp}$-module $x$ and some $k\geq 0$ (see for instance \cite[Lemma 2.1]{chen}) and since for every finitely generated $\Lambda^{\opp}$-module $x$ and all $l\geq 0$ it holds that $x\cong \Omega^l(x)[l]$ in the singularity category, we can restate the above in terms of modules; namely, a functor $F$ belongs to $Q^{\perp_{\gg}}$ if and only if for any finitely generated $\Lambda^{\opp}$-module $x$, viewed as an object in the singularity category, it holds that $F(\Omega^n(x))= 0$ for $n\gg 0$ (depending on $x$). When $F$ is an injective object in $\Mod\mathcal{D}$, it is also a cohomological functor (for instance by using the equivalence (\ref{equivalence})). Therefore, since every finitely generated module admits a composition series (alternatively since the singularity category is the thick closure of the simple modules), the last condition can be tested only on the simple $\Lambda^{\opp}$-modules; namely $F(\Omega^n(\Lambda^{\opp}/\mathsf{rad}\Lambda^{\opp}))=0$ for $n\gg0$.
\end{rem}



\subsection{Finitistic dimension along fully faithful functors} 

We use the characterisation of Theorem \ref{thmA} to study the finiteness of the finitistic dimension under fully faithful functors of singularity categories. For this reason, we recall that given a triangle functor $f\colon \mathcal{S}_1\rightarrow\mathcal{S}_2$ of small triangulated categories, there is an induced adjoint triple 
\[
\begin{tikzcd}
\mathsf{Mod}\mathcal{S}_2 \arrow[rr, "f^*"] &  & \mathsf{Mod}\mathcal{S}_1 \arrow[ll, "f^!"', bend right] \arrow[ll, "f_*", bend left]
\end{tikzcd}
\]
between categories of modules. 
\begin{lem} \label{lem: adjunction}
    Assuming the above setup, the following hold. 
    \begin{itemize}
        \item[\textnormal{(i)}] If $f$ is fully faithful, then $f^!$ and $f_*$ are fully faithful. 
        \item[\textnormal{(ii)}] If $f$ admits a right adjoint $g$, then $(f^*,g^*)$ is an adjoint pair. 
    \end{itemize}
\end{lem}
\begin{proof}
    This is well-known, see for instance \cite[Theorem 2.3.3]{sheaves}.
\end{proof}

\begin{thm} \label{ff}
    Let $\Gamma,\Lambda$ be Artin algebras and assume the existence of a fully faithful triangle functor $\mathsf{D}_{\mathsf{sg}}(\Gamma^{\opp})\hookrightarrow\mathsf{D}_{\mathsf{sg}}(\Lambda^{\opp})$ which has a right adjoint. If it holds that $\Findim\Lambda<\infty$, then $\Findim\Gamma<\infty$.  
\end{thm}
\begin{proof}
    We denote the given functor by $f$, its right adjoint by $g$ and furthermore we write $\mathcal{D}_{\Lambda}$ for $\mathsf{D}_{\mathsf{sg}}(\Lambda^{\opp})$, $\mathcal{D}_{\Gamma}$ for $\mathsf{D}_{\mathsf{sg}}(\Gamma^{\opp})$ and $Q_{\Lambda}$, $Q_{\Gamma}$ for their images along Yoneda, respectively. We have the following induced adjunctions 
    \[
\begin{tikzcd}
\mathsf{Mod}\mathcal{D}_{\Lambda} \arrow[rr, "f^*\simeq g^!"] \arrow[rr, "g_*"', bend right=60] &  & \mathsf{Mod}\mathcal{D}_{\Gamma} \arrow[ll, "f^!"', bend right] \arrow[ll, "f_*\simeq g^*"', bend left]
\end{tikzcd}
    \]
    where the identifications of functors follow from Lemma \ref{lem: adjunction}(ii). It holds that $g^!(Q_{\Lambda})\subseteq Q_{\Gamma}$ and, as a consequence of the adjunction $(g^!,g^*)$, we have an inclusion $g^*(Q_{\Gamma}^{\perp_{\gg}})\subseteq Q_{\Lambda}^{\perp_{\gg}}$, since $f^*$ and $g^*$ commute with the autoequivalences considered in the module categories. Moreover, $g^*$ maps injective objects to injective objects, being right adjoint to the exact functor $f^*$. By combining the above, it follows that 
    \begin{equation} \label{inclusion}
        g^*(Q_{\Gamma}^{\perp_{\gg}}\cap \Inj\mathcal{D}_{\Gamma})\subseteq Q_{\Lambda}^{\perp_{\gg}}\cap \Inj\mathcal{D}_{\Lambda}.
    \end{equation}
    Let now $x$ be an object in the intersection $Q_{\Gamma}^{\perp_{\gg}}\cap \Inj\mathcal{D}_{\Gamma}$. By using the assumption that $\Findim\Lambda<\infty$, together with Theorem \ref{thmA} and the inclusion (\ref{inclusion}), we see that $g^*(x)\cong 0$. Since $f$ is fully faithful, it follows from Lemma \ref{lem: adjunction}(i) that the functor $f_*\simeq g^*$ is also fully faithful, implying that $g^!g^*\simeq \mathsf{Id}_{\mathsf{Mod}\mathcal{D}_{\Gamma}}$. As a consequence, it holds that $x\cong g^!g^*(x)\cong 0$, thus $\Findim\Gamma<\infty$, again by invoking Theorem \ref{thmA}. 
\end{proof}

Our interest in Theorem \ref{ff} is structural, but it does yield concrete examples.

\begin{exmp}
    Assume that $\Gamma,\Lambda$ are finite dimensional algebras over a field $k$ and let $f\colon \Gamma\to \Lambda$ be a homomorphism of $k$-algebras satisfying that the cone $\mathsf{cone}(f)$ belongs to $ \mathsf{K}^{\mathsf{b}}(\proj \Gamma^\mathsf{e})$. Under the canonical isomorphism $(\Gamma^{\opp})^{\mathsf{e}}\cong \Gamma^{\mathsf{e}}$, we have $\mathsf{cone}(f^{\opp})\cong \mathsf{cone}(f)$ and therefore $f^{\opp}\colon\Gamma^{\opp}\rightarrow \Lambda^{\opp}$ satisfies the same condition. In particular, $\Lambda^{\opp}$ has finite projective dimension as a left and as a right $\Gamma^{\opp}$-module and so it follows (see for instance \cite[Proposition 3.3]{OPS}) that there is an induced adjunction
    \[
    \begin{tikzcd}
\mathsf{D}_{\mathsf{sg}}(\Lambda^{\opp}) \arrow[rr, "\mathsf{res}"'] &  & \mathsf{D}_{\mathsf{sg}}(\Gamma^{\opp}) \arrow[ll, "\Lambda\otimes^{\mathbb{L}}_{\Gamma}-"', bend right]
\end{tikzcd}
    \]
    of singularity categories. Then, \cite[Proposition 3.7]{OPS} asserts that the left adjoint is fully faithful, hence it follows from Theorem \ref{ff} that if $\Findim \Lambda<\infty$, then $\Findim\Gamma<\infty$. This applies in particular to `quotient bifinite extensions' in the sense of \cite{bifinite extensions}; compare with \cite[Theorem 1.1]{bifinite extensions}.
\end{exmp}

\subsection{Cohomological functors} \label{coh functors}

Before proving Theorem \ref{thmA'}, let us recall from \cite{rouquier} that a cohomological functor $F\colon \mathcal{D}^{\opp}\rightarrow \mathsf{Ab}$, where $\mathcal{D}$ is a triangulated category, is \emph{bounded below} if for every $x\in\mathcal{D}$, it holds that $F(x[n])=0$ for $n\ll 0$. This makes the condition of Theorem \ref{thmA'}, in combination with the computations of Remark \ref{remark}, equivalent to the condition $Q^{\perp_{\gg}}\cap \mathsf{Coh}\mathcal{D}=0$, where $\mathsf{Coh}\mathcal{D}$ denotes the additive functors $\mathcal{D}^{\opp}\rightarrow\mathsf{Ab}$ that are cohomological. As in Remark~\ref{remark}, it is enough to check the vanishing $F(x[n])=0$ only for $x=\Lambda^{\opp}/\mathsf{rad}\Lambda^{\opp}$. 

\begin{lem} \label{lem: coh vs inj}
    Let $\mathcal{T}$ be a compactly generated triangulated category and denote by $Q$ the image of $\mathcal{T}^{\mathsf{c}}$ inside $\Mod\mathcal{T}^{\mathsf{c}}$ along Yoneda. Then it holds that $Q^{\perp_{\gg}}\cap \Inj\mathcal{T}^{\mathsf{c}}=0$ if and only if $Q^{\perp_{\gg}}\cap \mathsf{Coh}\mathcal{T}^{\mathsf{c}}=0$.
\end{lem}
\begin{proof}
    One implication is clear, since $\Inj\mathcal{T}^{\mathsf{c}}\subseteq \mathsf{Coh}\mathcal{T}^{\mathsf{c}}$. Assume now that the intersection $Q^{\perp_{\gg}}\cap \Inj\mathcal{T}^{\mathsf{c}}$ is trivial and let $F$ be an object of $Q^{\perp_{\gg}}\cap \mathsf{Coh}\mathcal{T}^{\mathsf{c}}$. The latter is the same as a cohomological functor $(\mathcal{T}^{\mathsf{c}})^{\opp}\rightarrow  \mathsf{Ab}$ that is bounded below. We may then consider the character dual $F^+\colon \mathcal{T}^{\mathsf{c}}\rightarrow \mathsf{Ab}$ which is homological and therefore, by \cite{beligiannis, krause}, a flat object in $\Mod \ (\mathcal{T}^{\mathsf{c}})^{\opp}$ (see in particular \cite[Subsection 2.3]{krause}). As a consequence, the object $F^{++}$ belongs to $\Inj\mathcal{T}^{\mathsf{c}}$ and it is also bounded below. By assumption it follows that $F^{++}=0$, which yields $F=0$.
\end{proof}

\begin{proof}[Proof of Theorem \ref{thmA'}]
    If $Q^{\perp_{\gg}}\cap \mathsf{Coh}\mathcal{D}=0$, then $Q^{\perp_{\gg}}\cap \Inj\mathcal{D}=0$ and so it follows that $\Findim\Lambda<\infty$ by Theorem \ref{thmA}. On the other hand, if $\Findim\Lambda<\infty$, then it follows that $Q^{\perp_{\gg}}\cap \Inj\mathcal{D}=0$ and therefore $Q^{\perp_{\gg}}\cap \mathsf{Coh}\mathcal{D}=0$ by viewing $\mathcal{D}$ as the compact objects of $\mathsf{K}_{\mathsf{ac}}(\Inj\Lambda^{\opp})$ and using Lemma \ref{lem: coh vs inj}. 
\end{proof}

We close this paper with a remark on the work of Cummings \cite{cummings}. 

\begin{rem} \label{cummings}
    The work of Cummings \cite{cummings} on the symmetry of finite finitistic dimension has two interesting consequences: 
    \begin{itemize}
        \item[(i)] We know by \cite[Theorem 3.5]{cummings} that the implication $\Findim\Lambda<\infty \implies \Findim\Lambda^{\opp}<\infty$ holds for every finite dimensional algebra $\Lambda$ if and only if the FDC holds. Therefore, replacing $\mathsf{D}_{\mathsf{sg}}(\Lambda^{\opp})$ with $\mathsf{D}_{\mathsf{sg}}(\Lambda)$ in the statement of Theorem \ref{thmA} is equivalent to the FDC. This also hints that a version of Corollary \ref{singular invariant} that does \emph{not} use opposite algebras is out of reach (except for very specific singular equivalences where it is already known; see for instance \cite[Lemma 4.13]{wang}). 
        \item[(ii)] We further know by \cite[Theorem 3.5]{cummings} that it is enough to test the previous implication for finite dimensional algebras with zero finitistic dimension. Therefore the validity of the FDC depends only on the singularity categories of finite dimensional algebras with zero finitistic dimension; namely the FDC holds if and only if $\mathsf{D}_{\mathsf{sg}}(\Lambda^{\opp})$ satisfies the characterisation of Theorem \ref{thmA} (or Theorem \ref{thmA'}) for every finite dimensional algebra $\Lambda$ with $\Findim \Lambda^{\opp}=0$. 
    \end{itemize}
\end{rem}

\appendix 
\section{Non-existence of bounded t-structures}

The aim of this appendix is to study more systematically the arguments of Subsection \ref{coh functors}. In doing so, we show that the singularity category of a locally noetherian Grothendieck category, satisfying a technical condition, does not admit a bounded t-structure.

\begin{reminder}
    Let $(\mathcal{U},\mathcal{V})$ be a t-structure on a triangulated category $\mathcal{D}$ with heart $\mathcal{D}^{\heart}\coloneqq \mathcal{U}\cap \mathcal{V}[1]$. The t-structure is called \emph{bounded below} if $\mathcal{D}=\cup_{n\in\mathbb{Z}}\mathcal{V}[n]$, \emph{bounded above} if $\mathcal{D}=\cup_{n\in\mathbb{Z}}\mathcal{U}[n]$ and \emph{bounded} if it is both bounded below and bounded above. In the presence of a bounded below t-structure, the homological functor $H^0\colon \mathcal{D}\rightarrow \mathcal{D}^{\heart}$ satisfies that for any object $x$ of $\mathcal{D}$, it holds $H^0(x[n])=0$ for $n\ll 0$. Analogous conclusions hold for bounded above and bounded t-structures.
\end{reminder} 

\begin{rem}
    Let $\mathcal{D}$ be a triangulated category with a bounded t-structure $(\mathcal{U},\mathcal{V})$. Then the given t-structure is non-degenerate, meaning that 
    \[
    \cap_{n\in\mathbb{Z}}\mathcal{U}[n]=\cap_{n\in\mathbb{Z}}\mathcal{V}[n]=0.
    \]
    Most importantly, in this case, $\mathcal{D}$ is idempotent complete; see \cite[Theorem]{idempotent}.
\end{rem}

Given an additive category $\mathcal{C}$, an autoequivalence $\Sigma\colon\mathcal{C}\rightarrow \mathcal{C}$ and a collection of objects $\mathcal{X}\subseteq \mathcal{C}$, we define the following subcategories of $\mathcal{C}$:
\[
\mathcal{X}^{\perp_{\ll}}=\{y\in\mathcal{C}:\forall x\in\mathcal{X}, \mathsf{Hom}_{\mathcal{C}}(x,\Sigma^ny)=0 \ \text{for} \ n\ll 0\}.
\]
\[
\mathcal{X}^{\padova}=\{y\in\mathcal{C}:\forall x\in\mathcal{X}, \mathsf{Hom}_{\mathcal{C}}(x,\Sigma^ny)=0 \ \text{for} \ |n|\gg 0\}.
\]
In particular, for a triangulated category $\mathcal{T}$ and $\Sigma=[1]$, we may form $(\mathcal{T}^{\mathsf{c}})^{\perp_{\ll}}\eqqcolon\mathcal{T}^+$, the subcategory of \emph{bounded below} objects and $(\mathcal{T}^{\mathsf{c}})^{\padova}\eqqcolon \mathcal{T}^{\mathsf{b}}$, the subcategory of \emph{bounded objects} of $\mathcal{T}$; see \cite{KPV} for more details. From now on, $\mathcal{T}$ denotes a compactly generated triangulated category. As always we regard $\mathcal{T}^{\mathsf{c}}$ as a subcategory of $\Mod \mathcal{T}^{\mathsf{c}}$ which we denote by $Q$. 

\begin{thm} \label{thm: appendix thm}
    The following hold for $\mathcal{T}$, assuming that $\mathcal{T}\neq 0$.
    \begin{itemize}
        \item[\textnormal{(i)}] If $\mathcal{T}^-=0$, then $\mathcal{T}^{\mathsf{c}}$ does not admit a non-degenerate bounded below t-structure. 
        \item[\textnormal{(ii)}] If $\mathcal{T}^+=0$, then $\mathcal{T}^{\mathsf{c}}$ does not admit a non-degenerate bounded above t-structure. 
        \item[\textnormal{(iii)}] If $\mathcal{T}^{\mathsf{b}}=0$, then $\mathcal{T}^{\mathsf{c}}$ does not admit a bounded t-structure. 
    \end{itemize}
\end{thm}
\begin{proof}
    All the claims are similar; claim (iii) does not need a non-degeneracy assumption, since this is always satisfied for bounded t-structures. For this reason let us only show (i). Since $\mathcal{T}^-=0$, it follows also that $\mathcal{T}^-\cap \mathsf{PInj}\mathcal{T}=0$. Thus, we know from Lemma \ref{lem: bounded above pure} that $Q^{\perp_{\gg}}\cap \Inj\mathcal{T}^{\mathsf{c}}=0$ and so $Q^{\perp_{\gg}}\cap \mathsf{Coh}\mathcal{T}^{\mathsf{c}}=0$, by invoking Lemma~\ref{lem: coh vs inj}. As a consequence, there are no nonzero bounded below cohomological functors on $\mathcal{T}^{\mathsf{c}}$ (that is $Q^{\perp_{\gg}}\cap \mathsf{Coh}\mathcal{T}^{\mathsf{c}}=0$). Assume that $\mathcal{T}^{\mathsf{c}}$ admits a non-degenerate bounded below t-structure with heart $\mathcal{H}$. Since $\mathcal{T}^{\mathsf{c}}$ is small, so is $\mathcal{H}$ and by the Freyd-Mitchell theorem there is a fully faithful exact functor $\mathcal{H}\rightarrow \Mod R$ for some associative ring $R$. Thereby, we can consider the following composite 
    \[
    G\colon \mathcal{T}^{\mathsf{c}}\xrightarrow{H^0}\mathcal{H}\rightarrow \Mod R\rightarrow \mathsf{Ab}, 
    \]
    where $\Mod R\rightarrow \mathsf{Ab}$ is simply the forgetful functor. Since the t-structure is assumed to be bounded below, it holds that $H^0(x[n])=0$ for every $x\in\mathcal{T}^{\mathsf{c}}$ and $n\ll 0$. We thus infer that $G^+$ belongs to $Q^{\perp_{\gg}}\cap \mathsf{Coh}\mathcal{T}^{\mathsf{c}}$. As a consequence, $G^+=0$ and therefore $G=0$ too. Since the t-structure is non-degenerate, it cannot happen that $H^0$ vanishes - a contradiction. 
\end{proof}

\begin{cor} \label{no bounded below}
    Let $\Lambda$ be an Artin algebra of infinite global dimension satisfying $\Findim\Lambda^{\opp}<\infty$. Then there does not exist a non-degenerate bounded below t-structure on $\mathsf{D}_{\mathsf{sg}}(\Lambda)$. 
\end{cor}
\begin{proof} 
    We identify $\mathsf{D}_{\mathsf{sg}}(\Lambda)$ with the compact objects of $\mathcal{T}=\mathsf{K}_{\mathsf{ac}}(\Inj\Lambda)$. Since $\Findim\Lambda^{\opp}<\infty$, it follows from Lemma \ref{lem: computation from KPV} and Proposition \ref{prop: theorem of shaul} that $\mathcal{T}^-=0$, so Theorem \ref{thm: appendix thm} applies. 
\end{proof}

\begin{rem} \label{rem: singularity of Artin}
    That $\mathsf{D}_{\mathsf{sg}}(\Lambda)$, where $\Lambda$ is an Artin algebra of infinite global dimension, does not admit bounded t-structures was already known, without finiteness of the finitistic dimension. Indeed, in the presence of a bounded t-structure on $\mathsf{D}_{\mathsf{sg}}(\Lambda)$, it would follow that $\mathsf{Hom}_{\mathsf{D}_{\mathsf{sg}}(\Lambda)}(x,y[n])=0$ for $n\ll 0$ and for any pair of objects $x,y$ of $\mathsf{D}_{\mathsf{sg}}(\Lambda)$. However, we know from \cite{chen_li_zhang_zhao} that $\mathsf{Hom}_{\mathsf{D}_{\mathsf{sg}}(\Lambda)}(s,s[n])\neq 0$ for all $n\in\mathbb{Z}$, where $s=\Lambda/\mathsf{rad}\Lambda$. 
\end{rem}

We now wish to understand examples other than Artin algebras. For a locally noetherian Grothendieck category $\mathcal{A}$, we denote by $\mathsf{noeth}\mathcal{A}$ its subcategory that consists of noetherian objects. We will assume the following condition is met 
\begin{equation} \label{*}
    \mathsf{D}^{\mathsf{b}}(\mathcal{A})=\mathsf{broad}_{\Sigma}(\mathsf{D}^{\mathsf{b}}(\mathsf{noeth}\mathcal{A})),
\end{equation}
where the right-hand side denotes the smallest triangulated subcategory of $\mathsf{D}(\mathcal{A})$ that contains $\mathsf{D}^{\mathsf{b}}(\mathsf{noeth}\mathcal{A})$ and is closed under self-coproducts; see \cite[Lemma 2.6]{KPV} for an explicit description.  

\begin{prop} \label{no bounded objects}
    Let $\mathcal{A}$ be a locally noetherian Grothendieck category that satisfies \textnormal{(\ref{*})} and such that $\mathsf{D}(\mathcal{A})$ is compactly generated. Then the equality $\mathcal{T}^{\mathsf{b}}=0$ holds for $\mathcal{T}=\mathsf{K}_{\mathsf{ac}}(\Inj\mathcal{A})$.
\end{prop}
\begin{proof}
    The proof follows \cite[Proposition 3.21]{KPV}, but let us nonetheless sketch the key points. By \cite{krause stable}, we know that there exists a recollement of triangulated categories 
    \[
    \begin{tikzcd}
\mathsf{K}_{\mathsf{ac}}(\Inj\mathcal{A}) \arrow[rr, "I"] &  & \mathsf{K}(\Inj\mathcal{A}) \arrow[ll, "I_{\lambda}"', bend right] \arrow[ll, "I_{\rho}", bend left] \arrow[rr, "Q"] &  & \mathsf{D}(\mathcal{A}) \arrow[ll, "Q_{\lambda}"', bend right] \arrow[ll, "Q_{\rho}", bend left]
\end{tikzcd}
    \]
    and $\mathsf{K}_{\mathsf{ac}}(\Inj\mathcal{A})$ is compactly generated with $\mathsf{K}_{\mathsf{ac}}(\Inj\mathcal{A})^{\mathsf{c}}= \mathsf{thick}(I_{\lambda}Q_{\rho}(\mathsf{D}^{\mathsf{b}}(\mathsf{noeth}\mathcal{A})))$. Since condition (\ref{*}) is satisfied, it follows from \cite[Lemma 2.9]{KPV} that 
    \[
    \mathcal{T}^{\mathsf{b}}=I_{\lambda}Q_{\rho}(\mathcal{A})^{\padova}.
    \]
    By the adjunction $(I_{\lambda},I)$ it follows that an object $y$ belongs to $\mathcal{T}^{\mathsf{b}}$ if and only if 
    \[
    \mathsf{Hom}_{\mathsf{K}(\Inj\mathcal{A})}(Q_{\rho}(x),y[n])\cong \mathsf{Hom}_{\mathsf{K}(\mathcal{A})}(x,y[n])\cong \mathsf{Ext}_{\mathcal{A}}^1(x,Z^{n-1}_y)
    \]
    vanishes for every $x\in\mathcal{A}$ and $|n|\gg 0$ (note that $Q_{\rho}$ maps objects to injective resolutions, hence the first isomorphism; see \cite[Lemma 2.1]{krause stable}). We then observe that we can fix a negative integer $k$ such that $\mathsf{Ext}^1_{\mathcal{A}}(x,Z^{n-1}_y)=0$ for all $n\leq  k$. Indeed, otherwise, we may find objects $x_l$ and a strictly decreasing sequence of negative integers $k_l$ such that $\mathsf{Ext}^1_{\mathcal{A}}(x_l,Z^{k_l-1}_y)\neq 0$. Then for $x=\oplus x_l$ it would hold that $\mathsf{Ext}^1_{\mathcal{A}}(x,Z^{n-1}_y)\neq 0$ for arbitrarily small values of $n$. This argument shows that the complex $y$ splits in small degrees and since it is already an acyclic complex of injectives, it is nullhomotopic. 
\end{proof}

\begin{cor} \label{cor1}
    Let $\mathcal{A}$ be a locally noetherian Grothendieck category satisfying the assumptions of Proposition \ref{no bounded objects}. Then $\mathsf{D}_{\mathsf{sg}}(\mathcal{A})$, the singularity category of $\mathcal{A}$ defined as $\mathsf{D}^{\mathsf{b}}(\mathsf{noeth}\mathcal{A})/\mathsf{D}(\mathcal{A})^{\mathsf{c}}$, does not admit a bounded t-structure (assuming that $\mathsf{D}_{\mathsf{sg}}(\mathcal{A})\neq 0)$. 
\end{cor}
\begin{proof}
    First of all, we know from \cite[Corollary 5.4]{krause stable} that $\mathsf{D}_{\mathsf{sg}}(\mathcal{A})^{\natural}\simeq \mathcal{T}^{\mathsf{c}}$, where $\mathcal{T}=\mathsf{K}_{\mathsf{ac}}(\Inj\mathcal{A})$. If $\mathsf{D}_{\mathsf{sg}}(\mathcal{A})$ had a bounded t-structure, then it would be idempotent complete and therefore $\mathsf{D}_{\mathsf{sg}}(\mathcal{A})\simeq \mathcal{T}^{\mathsf{c}}$. We thus see that Theorem \ref{thm: appendix thm} and Proposition \ref{no bounded objects} yield the claim.  
\end{proof}

One now wonders whether (\ref{*}) is ever true. For Artin algebras this is quite simple, see \cite[Corollary 2.8]{KPV}. It is more subtle that the same condition is true for noetherian quasiexcellent rings of finite Krull dimension. The argument follows results from \cite{aoki,neeman2} and we refer the reader to \cite[Example 3.32]{KPV} and the precise references therein. The point is that the following holds.

\begin{cor} \label{cor2}
    Let $R$ be a commutative noetherian quasiexcellent ring of finite Krull dimension that is not regular. Then the singularity category $\mathsf{D}_{\mathsf{sg}}(R)$ does not admit a bounded t-structure.  
\end{cor}

\end{document}